\documentclass[12pt]{amsart}

\usepackage{amssymb,amsmath}
\usepackage{amssymb,latexsym}
\usepackage{amsfonts}
\usepackage{amssymb}
\usepackage{lscape}
\usepackage{stmaryrd}
\usepackage{longtable}
\usepackage{hyperref}
\usepackage{amsaddr}
\usepackage{amsmath, geometry, amssymb} 
\numberwithin{equation}{section}
\newtheorem{theorem}{Theorem}[section]
\newtheorem{lemma}{Lemma}[section]

\numberwithin{table}{section}

\DeclareMathOperator{\lcm}{lcm}

\newcommand{\sqr}[2]{{\vcenter{\vbox{\hrule height#2pt
 \hbox{\vrule width#2pt height#1pt \kern#1pt
 \vrule width#2pt}\hrule height#2pt}}}}

\newcommand{\beq}{\begin{equation}}
\newcommand{\eeq}{\end{equation}}
\newcommand{\beqar}{\begin{eqnarray}}
\newcommand{\eeqar}{\end{eqnarray}}
\def\beqars{\begin{eqnarray*}}
\def\eeqars{\end{eqnarray*}}

\newcommand{\cc}{\mathbb{C}}

\allowdisplaybreaks

\begin{document} 

\markboth{Zafer Selcuk Aygin}
{On Eisenstein series}

\title{On Eisenstein series in $M_{2k}(\Gamma_0(N))$ and their applications}

\author{Zafer Selcuk Aygin}
\address{Division of Mathematical Sciences, School of Physical and Mathematical Sciences, Nanyang Technological University, 21 Nanyang Link, Singapore 637371, Singapore} 
\email{selcukaygin@ntu.edu.sg}

\maketitle

\begin{abstract}
Let $k,N \in \mathbb{N}$ with $N$ square-free and $k>1$. We prove an orthogonal relation and use this to compute the Fourier coefficients of the Eisenstein part of any $f(z) \in M_{2k}(\Gamma_0(N))$ in terms of sum of divisors function. In particular, if $f(z) \in E_{2k}(\Gamma_0(N))$, then the computation will to yield to an expression for the Fourier coefficients of $f(z)$. Then we apply our main theorem to give formulas for convolution sums of the divisor function to extend the result by Ramanujan, and to eta quotients which yields to formulas for number of representations of integers by certain families of quadratic forms. At last we give essential results to derive similar results for modular forms in a more general setting.

\end{abstract}


\section{Introduction}\label{sec:1}
Let $\mathbb{N}$, $\mathbb{N}_0$, $\mathbb{Z}$, $\mathbb{Q}$, $\mathbb{C}$ and $\mathbb{H}$ denote the sets of positive integers, non-negative integers, integers, rational numbers, complex numbers and the upper half plane, respectively. Throughout the paper we let $z \in \mathbb{H}$ and  $q=e^{2 \pi i z}$. Let $\Gamma_0(N)$ ($N \in \mathbb{N}$) be  the modular subgroup defined by
\begin{align*}
\Gamma_0(N) = \left\{ \begin{bmatrix}
    a       & b  \\
    c       & d 
\end{bmatrix}  \mid  a,b,c,d\in \mathbb{Z} ,~ ad-bc = 1,~c \equiv 0 \pmod {N}
\right\} .
\end{align*} 
An element {\renewcommand*{\arraystretch}{0.5} $M= \begin{bmatrix}
    a       & b  \\
    c       & d 
\end{bmatrix} \in \Gamma_0(1)$} acts on $\mathbb{H} \cup \mathbb{Q} \cup \{ \infty \}$ by
\begin{align*}
\displaystyle M(z)=\left\{\begin{array}{ll}
	\frac{az+b}{cz+d} & \mbox{ if } z \neq \infty, \\
	\frac{a}{c} & \mbox{ if } z= \infty .
\end{array} \right.
\end{align*}

Let $k, N \in \mathbb{N}$. We write $M_k(\Gamma_0(N))$ to denote the space of modular forms of weight $k$ for  $\Gamma_0(N)$; and 
$E_k (\Gamma_0(N))$ and $S_k(\Gamma_0(N))$ to denote the subspaces of Eisenstein forms and cusp forms 
of  $M_k(\Gamma_0(N))$, respectively. 
It is known  (see for example \cite{Serre} and \cite[p. 83]{stein}) that
\begin{align}
M_k (\Gamma_0(N)) = E_k (\Gamma_0(N)) \oplus S_k(\Gamma_0(N)). \label{eq:10}
\end{align}

In the remainder of the paper, unless otherwise stated, we assume $N\in \mathbb{N}$ is square-free, $p$ always stands for prime numbers, and all the divisors considered are positive divisors.

A set of representatives of cusps of $\Gamma_0(N)$, when $N$ is square-free, is given by
\begin{align*}
R(N)=\left\{ \frac{1}{c} : c \mid N  \right\},
\end{align*}
see \cite[Proposition 2.6]{iwaniec}.

Let 
\begin{align} 
A_c=\begin{bmatrix}
    -1      & 0  \\
    c       & -1 
\end{bmatrix}, \label{eq:15}
\end{align}
then the Fourier series expansion of $f(z) \in M_k(\Gamma_0(N))$ at the cusp $ \frac{1}{c}\in \mathbb{Q} \cup \{\infty \}$ is given by the Fourier series expansion of $f(A_c^{-1}z)$ at the cusp $\infty$, see \cite[pg. 35]{Kohler}. Let the Fourier series expansion of $f(z)$ at the cusp $ \frac{1}{c} $ be given by the infinite sum
\begin{align*}
f(A_c^{-1}z)=(cz+1)^k \sum_{n\geq 0} a_c(n) e^{2 \pi i n z/h},
\end{align*}
where $h$ is the width of $\Gamma_0(N)$ at the cusp $\frac{1}{c}$. Then we use the notation $[n]_cf(z)$ to denote $a_c(n)$. We write $[n]$, instead of $[n]_0$, at the cusp $\infty ~(=1/0)$. If we say Fourier series expansion (or Fourier coefficients) without specifying the cusp, we mean the expansion (or coefficients) at the cusp $\infty$. And, for modular forms, `the first term of the Fourier expansion of $f(z)$ at cusp $\frac{1}{c}$' refers to the term $[0]_cf(z)$. We define $v_c(f)$ the order of $f(z)$ at $\frac{1}{c}$ to be the smallest $n$ such that $[n]_cf(z) \neq 0$.

The Dedekind eta function $\eta (z)$ is the holomorphic function defined on the upper half plane $\mathbb{H} = \{ z \in \mathbb{C} \mid \mbox{\rm Im}(z) >0 \}$ 
by the product formula
\begin{align*}
\eta (z) = e^{\pi i z/12} \prod_{n=1}^{\infty} (1-e^{2\pi inz}). 
\end{align*}
Let $N \in \mathbb{N}$ and $r_\delta \in \mathbb{Z}$ for all $1\leq \delta \mid N$, where the $r_\delta$ are not all zero. Let $z \in \mathbb{H}$. We define an eta quotient by the product formula
\begin{align}
f(z) = \prod_{ \delta \mid N} \eta^{r_{\delta}} ( \delta z). 
\end{align} 

For $k \in \mathbb{N}_0$ and $n \in \mathbb{N}$ we define the sum of divisors function by
\begin{align*}
\sigma_k(n)=\sum_{1 \leq d \mid n}d^k.
\end{align*}
If $n \notin \mathbb{N}$ we set $\sigma_k(n)=0$. We define the  Eisenstein series $E_{2k}(z)$ by
\begin{align} \label{1_80}
&&\displaystyle E_{2k} (z)  :=1-\frac{4k}{B_{2k}}\sum_{n=1}^{\infty} \sigma_{2k-1}(n)q^{ n },
\end{align}
where $B_{2k}$ are Bernoulli numbers, defined by the generating function
\begin{align*} 
\displaystyle \frac{x}{e^x-1}=\sum_{n=0}^\infty\frac{B_n x^n}{n!}.
\end{align*}
For all $d \mid N$ and $1 < k \in \mathbb{N}$ we have 
\begin{align*}
&& E_{2k}(dz) \in E_{2k}(\Gamma_0(N)),  
\end{align*}
see \cite[Theorem 5.8]{stein}. Noting that $\omega(d)$ stands for the number of distinct prime divisors of $d$, we state our main theorem.
\begin{theorem} \label{th:1}
Let $N\in \mathbb{N}$ be square-free and $k>1$ be an integer. Let $f(z)\in M_{2k}(\Gamma_0(N))$. Then there exists a cusp form $C_{N}(f;z) \in S_{2k}(\Gamma_0(N))$ such that 
\begin{align}
& f(z)=\sum_{d \mid N} \left( \sum_{c \mid N} \frac{(-1)^{\omega(d)+\omega(c)} }{\prod_{p \mid N} (p^{2k}-1)} \left( \frac{ N\gcd(c,d)}{c}  \right)^{2k} [0]_c f(z) \right) E_{2k}(dz) + C_{N}(f;z). \label{eq:2}
\end{align}
In particular, if $f(z)\in E_{2k}(\Gamma_0(N))$ then we have 
\begin{align}
f(z)=\sum_{d \mid N}  \left( \sum_{c \mid N} \frac{(-1)^{\omega(d)+\omega(c)} }{\prod_{p \mid N} (p^{2k}-1)} \left( \frac{ N\gcd(c,d)}{c}  \right)^{2k} [0]_c f(z) \right) E_{2k}(dz). \label{eq:3}
\end{align}
Comparing coefficients of $q^n$ on both sides of equations {\rm (\ref{eq:2})} and {\rm (\ref{eq:3})}, for $n>0$ we have
{\small \begin{align}
& [n]f(z)=\frac{-4k}{B_{4k}}\sum_{d \mid N}  \left( \sum_{c \mid N} \frac{(-1)^{\omega(d)+\omega(c)} }{\prod_{p \mid N} (p^{2k}-1)} \left( \frac{ N\gcd(c,d)}{c}  \right)^{2k} [0]_c f(z) \right) \sigma_{2k-1}(n/d) + [n] C_{N}(f;z),\label{eq:47}
\end{align}}%
in particular, if $f(z)\in E_{2k}(\Gamma_0(N))$ then for $n>0$ we have 
\begin{align}
[n]f(z)=\frac{-4k}{B_{4k}}\sum_{d \mid N}  \left( \sum_{c \mid N} \frac{(-1)^{\omega(d)+\omega(c)} }{\prod_{p \mid N} (p^{2k}-1)} \left( \frac{ N\gcd(c,d)}{c}  \right)^{2k} [0]_c f(z) \right) \sigma_{2k-1}(n/d). \label{eq:32}
\end{align}
\end{theorem}

If $f(z) \in M_{2k}(\Gamma_0(N))$, then it is analytic at all cusps. Thus, it is possible to calculate $[0]_cf(z)$ for all $c\mid N$. That is, Theorem \ref{th:1} can be applied to any $f(z) \in M_{2k}(\Gamma_0(N))$ to obtain its Eisenstein part. In most applications of modular forms in number theory, we realized, only the expansion at infinity is considered. This approach is pretty useful, however it fails to provide general results. Consideration of expansions at other cusps, as we do in this paper, allows us to derive more general results using modular forms. As detailed in Section \ref{sec:8} similar formulas can be obtained for other modular form spaces. This can be applied to a plethora of different questions, most notably to deriving arithmetic properties of modular forms in a more general setting, see Section \ref{sec:8}.

In the next section, we compute Fourier series expansions of the modular forms $f(dz)$ where $f(z) \in M_{2k}(\Gamma_0(1))$ at the cusps $\frac{1}{c} \in \mathbb{Q}$. In Section \ref{sec:2}, we use these expansions, together with the orthogonal relation given by Lemma \ref{lemma:1}, to prove Theorem \ref{th:1}. In Sections \ref{sec:4} and \ref{sec:7}, we apply Theorem \ref{th:1}, to convolution sums of the divisor function and eta quotients, respectively. In the last section we describe how to derive statements similar to Theorem \ref{th:1} for other modular form spaces. All the theorems and corollaries stated in this paper are new, and if we fix $N$ and/or $k$, our formulas agree with the previously known formulas. We discuss the details of Sections \ref{sec:4}, \ref{sec:7} and \ref{sec:8}.

Let $a,b,l,m \in \mathbb{N}$ and let us define the convolution sum
\begin{align*}
W(a^{l},b^{m};n)= \sum_{ar+bs=n} \sigma_{l}(r) \sigma_{m}(s)
\end{align*}
for $n>0$. This convolution sum is a generalized version of the convolution sum defined by Ramanujan in \cite{19ramanujanocaf}, where he gave a formula for $W(1^{2l-1},1^{2m-1};n)$. For some history on the formulas for $W(a^l,b^m;n)$, see \cite{aygineisandconv,history2}.
In Section \ref{sec:4}, we use Theorem \ref{th:1} to give the following formula:
{\footnotesize \begin{align}
& W(a^{2l-1},b^{2m-1};n)     = \label{eq:31} \\
& \quad \frac{-kB_{2l}B_{2m}}{4lm B_{2k}}\sum_{d \mid N} \left( \sum_{c \mid N}  \frac{(-1)^{\omega(d)+{\omega(c)}}}{\prod_{p \mid N} (p^{2k}-1)} \left( \frac{ N \gcd(c,d)}{c}  \right)^{2k} \left( \frac{\gcd(a,c)}{a} \right)^{2l} \left( \frac{\gcd(b,c)}{b} \right)^{2m}  \right) \sigma_{2k-1}(n/d)\nonumber \\
& \quad +\frac{B_{2l}}{4l} \sigma_{2m-1}(n/b)  + \frac{B_{2m}}{4m}  \sigma_{2l-1}(n/a) + \frac{B_{2l}B_{2m}}{16lm} [n]C_{\lcm(a,b)}(E_{2l}(az)E_{2m}(bz);z). \nonumber
\end{align}}%
where $a,b \in \mathbb{N}$ are square-free, $l,m>1$, $k=l+m$. This formula extends the result given by Ramanujan in \cite{19ramanujanocaf}, except the cases $l=1$ or $m=1$. This is due to the complicated behaviour of weight $2$ Eisenstein series at cusps. (See \cite{aygineisandconv}, for a treatment of the case when $l=1$ and $m=1$.) At the end of Section \ref{sec:4}, we let the level to be $6$ to illustrate this formula, we then describe $C_{\lcm(a,b)}(E_{2l}(az)E_{2m}(bz);z)$ (for all $\lcm(a,b)\mid 6$) in terms of eta quotients and finally we let $a,b=1$ to deduce that  (\ref{eq:31}) agrees with the formula given by Ramanujan.

Let $m \in \mathbb{N}$, $r_i \in \mathbb{N}_0$, and $a_i \in \mathbb{N}$ for all $1 \leq i \leq m$. Let 
\begin{align*}
N(a_1^{r_1},a_2^{r_2},\ldots,a_{m}^{r_m}; n)
\end{align*}
denote the number of representations of $n$ by the quadratic form
\begin{align*}
\sum_{i=1}^{m} \sum_{j=1}^{r_i} a_{i}x_j^2.
\end{align*}
In Section \ref{sec:7}, we apply Theorem \ref{th:1} to obtain the following formula for representations of an integer by a family of quadratic forms with odd square-free coefficients:
{\scriptsize \begin{align}
& N(1^{8r_1},\ldots,\delta^{8r_\delta},\ldots,N^{8r_N};n) =  \label{eq:40} \\
& \qquad  \frac{4k}{B_{2k}} \sum_{d \mid  N} \left( \sum_{c \mid  N} \frac{(-1)^{\omega(d)+\omega(c)+1} }{\prod_{p \mid 2N} (p^{2k}-1)} \left( \frac{ 2 N\gcd(c,d)}{c}  \right)^{2k} \prod_{ \delta \mid N} \left(\frac{\gcd(c,\delta)}{\delta}\right)^{4 r_{\delta}}  \right) \sigma_{2k-1}(n/{2d}) \nonumber \\
&\qquad + \frac{(-1)^n4k}{B_{2k}}  \sum_{ d \mid N} \left( \sum_{c \mid  N} \frac{(-1)^{\omega(d)+\omega(c)} }{\prod_{p \mid 2N} (p^{2k}-1)} \left( \frac{ N\gcd(c,d)}{c}  \right)^{2k} \prod_{ \delta \mid N} \left(\frac{\gcd(c,\delta)}{\delta}\right)^{4 r_{\delta}}  \right) \sigma_{2k-1}(n/d) \nonumber \\
& \qquad + (-1)^n [n]C_{2N}\left(\prod_{ \delta \mid N} g^{8r_\delta}(\delta z);z \right). \nonumber
\end{align}}%
Then we compare our formula with recent formulas from Cooper et al. \cite{cooperrmf} and classical results by Ramanujan \cite{cooper,mordell,19ramanujanocaf}. We also illustrate our formula with $2N=6$ and give a description of $C_{6}\left(g(z)^{8 r_1} g(3 z)^{8 r_3};z\right)$ in terms of eta quotients. 

One can use (\ref{eq:32}) to determine Fourier coefficients of eta quotients in $E_{2k}(\Gamma_0(N))$. In \cite{alacaaygin2015} we take advantage of a similar idea, for fixed weight and level, to determine Fourier coefficients of certain families of weight $2$ eta quotients. It seems, if $N \in \mathbb{N}$ square-free and the weight is greater than $2$, then there are only two eta quotients in $E_{2k}(\Gamma_0(N))$. We try to explain the reason for this in Section \ref{sec:7}.

In Section \ref{sec:8} we find the equivalent of (\ref{eq:19}) for Eisenstein series those twisted by Dirichlet characters. The rest of the section is dedicated to explain how to apply these methods to modular form spaces in a more general setting.

\section{Preliminary results}\label{sec:5}

Let $f(z) \in M_{2k}(\Gamma_0(1))$. In this section we state and prove Theorem \ref{th:3}, which gives the Fourier series expansions of $f(dz)\in M_{2k}(\Gamma_0(d))$ at cusps $\frac{1}{c} \in \mathbb{Q}$. We use Theorem \ref{th:3} to compute first terms of Fourier series expansions of $E_{2k}(dz)$ at the cusps $1/c$. Together with the fact that the first terms of Fourier series expansions of cusp forms are always $0$, Theorem \ref{th:3} will be used to prove Theorem \ref{th:1}. Then we give a `Sturm bound' for Eisenstein series and at last we give an interesting relationship between the Fourier coefficients of $f(z) \in E_{2k}(\Gamma_0(N))$ and $v_c(f)$.

\begin{theorem} \label{th:3}  
Let $k \in \mathbb{N}$. Let $f(z) \in M_{2k}(\Gamma_0(1))$, with the Fourier series expansion given by
\begin{align*}
f(z)=\sum_{n \geq 0} a_n q^n.
\end{align*}
Then for $d \in \mathbb{N}$, the Fourier series expansion of $f_d(z)=f(dz)$ at cusp $1/c\in \mathbb{Q}$ is given by
\begin{align*}
 f_{d}(A_c^{-1}z)=\Big(\frac{g}{d}\Big)^{2k}(cz+1)^{2k} f\Big(\frac{g^2}{d}z+\frac{y g}{d}\Big)=\Big(\frac{g}{d}\Big)^{2k}(cz+1)^{2k}\sum_{n \geq 0} a_n q_c^n,
\end{align*}
where $g=\gcd(d, c)$,  $y$ is some integer, $A_c$ is the matrix given by {\em (\ref{eq:15})} and $\displaystyle q_c=e^{2\pi i \left(\frac{g^2}{d}z+\frac{y g}{d}\right)}$.
\end{theorem}
\begin{proof}
The Fourier series expansion of $f_{d}(z)$ at the cusp $ 1/c $ is given by 
the Fourier series expansion of $f_{d}(A_c^{-1}z)$ at the cusp $\infty$. We have
\begin{align*}
f_{d}(A_c^{-1}z) =f_{d}\left(\frac{-z}{-cz-1}\right) = f\left(\frac{-dz}{-cz-1}\right) = f(\gamma z),
\end{align*}
where $\gamma = \begin{bmatrix}
    -d       & 0  \\
    -c       & -1 
\end{bmatrix}$. 
As $\gcd(d/g,c/g)=1$, there exist  $y,v\in \mathbb{Z}$ such that $\displaystyle \frac{-dv}{g}+\frac{cy}{g}=1$. 
Thus $L := \begin{bmatrix}
    -d/g       & y  \\
     -c/g     & v 
\end{bmatrix} \in SL_2(\mathbb{Z})$. 
Then for $k \geq 1$, we have
\begin{align*}
  f_{d}(A_c^{-1}z) =f(L L^{-1} \gamma z)&=\Big(-c\Big(\frac{(-vd +cy) z +y}{d}\Big) + v \Big)^{2k} f\Big( \frac{(-vd +cy) z +y}{d/g}\Big)\\
 &=(g/d)^{2k}\Big(c\frac{vd-cy}{g}z + \frac{vd-cy}{g} \Big)^{2k} f\Big( \frac{g^2 z + yg}{d}\Big)\\
 &=(g/d)^{2k}(cz + 1 )^{2k} f \Big( \frac{g^2 }{d}z + \frac{yg}{d}\Big),
\end{align*}
which completes the proof. Note that, the value of $\displaystyle e^{2 \pi i \frac{yg}{d} }$ is independent of choice of $y$.
\end{proof}

Noting that the width of $\Gamma_0(N)$ at the cusp $\frac{1}{c}$ is $\frac{N}{c}$, it follows from Theorem \ref{th:3} and (\ref{1_80}) that the terms of the Fourier coefficients of $E_{2k}(dz) \in M_{2k}(\Gamma_0(N))$ at the cusp $1/c$ ($c \mid N$) are
\begin{align}
 [0]_{c}E_{2k}(dz)=&\left(\frac{\gcd(c,d)}{d}\right)^{2k} , \label{eq:19}\\
 [n]_{c}E_{2k}(dz)=&\frac{-4k}{B_{2k}} e^{ \frac{2 \pi i ycn}{\gcd(c,d)N}} \left(\frac{\gcd(c,d)}{d}\right)^{2k}  \sigma_{2k-1}\left(\frac{cd}{\gcd(c,d)^2N} n \right) \mbox{, ($n\geq 1$)}, \label{eq:34}
\end{align}
for all $k > 1$, $d \in \mathbb{N}$. 

Let $k>1$ be an integer, $N\in \mathbb{N}$ be square-free and $p_1$ be the smallest prime that divides $N$. Then, using (\ref{eq:34}), it is not hard to see that
\begin{align}
0 = [n]_c \sum_{d \mid N} a_d E_{2k}(dz), \mbox{ $\displaystyle n \leq \frac{N}{p_1}$}, \label{eq:13}
\end{align}
if and only if $a_d=0$ for all $d \mid N$. This gives us the following theorem, which could be viewed as a Sturm Theorem for Eisenstein forms. Note that this bound is much smaller than the Sturm bound, which is to be expected.
\begin{theorem} \label{th:5}
Let $k>1$ be an integer, $N\in \mathbb{N}$ be square-free and $p_1$ be the smallest prime that divides $N$. Let $f(z) \in E_{2k}(\Gamma_0(N))$ be a non-zero function.  Then for all $c \mid N$, we have
$ \displaystyle v_{c}(f) \leq \frac{N}{p_1}.$
\end{theorem}

To stress the intimate relationship between a modular form and its orders of zeros at cusps we note the following relation. Let $N\in \mathbb{N}$ be square-free and $k>1$ be an integer, and $f(z)=\sum_{d \mid N}a_d E_{2k}(dz) \in E_{2k}(\Gamma_0(N))$. Then from (\ref{eq:34}), we deduce that if $v_c(f)>1$ then $a_{N/c}=0$.

\section{Proof of Theorem \ref{th:1}}\label{sec:2}
Let $N\in \mathbb{N}$ be square-free and $k>1$ be an integer. Assume that $f(z)\in M_{2k}(\Gamma_0(N))$. By \cite[Theorem 5.9]{stein} and (\ref{eq:10}), we have 
\begin{align*}
f(z)=\sum_{d \mid N} a_d E_{2k}(dz) + C_{N}(f;z), 
\end{align*}
for some $a_d \in \mathbb{C}$ and $C_{N}(f;z) \in S_{2k}(\Gamma_0(N))$. By definition the first terms of Fourier series expansions of cusp forms are always $0$. Then by (\ref{eq:19}), for each $c \mid N$, we have
\begin{align}
& [0]_c f(z)= \sum_{d \mid N} a_d  \left(\frac{\gcd(c,d)}{d}\right)^{2k} +0. \label{eq:11}
\end{align}
In Lemma \ref{lemma:1} below we give an orthogonal relation, which is useful for computing the inverse of the coefficient matrix of system of linear equations given by (\ref{eq:11}). Then we use this inverse matrix obtained to compute $a_d$ in terms of $[0]_c f(z)$. This completes the proof of Theorem \ref{th:1}.
\begin{lemma} {\label{lemma:1}}
Let $N\in \mathbb{N}$ be square-free and $k>1$ be an integer. Then for all $c,d \mid N$, we have
\begin{align*}
\displaystyle \sum_{t \mid N} (-1)^{\omega(t)+\omega(c)} \left( \frac{N\gcd(c,t)}{t} \cdot  \frac{\gcd(t,d)}{d} \right)^{2k}=\left\{\begin{array}{ll}
\displaystyle {\displaystyle \prod_{p \mid N} (p^{2k}-1)} & \mbox{ if } c=d , \\
0  & \mbox{ if } c \neq d .
\end{array} \right.
\end{align*}
\end{lemma}
\begin{proof} 
Let $N\in \mathbb{N}$ be square-free, $k>1$ be an integer and $c,d \mid N$. Let 
\begin{align*}
\displaystyle T(c,d;N)= \sum_{t \mid N} (-1)^{\omega(t)+\omega(c)} \left( \frac{N\gcd(c,t)}{t} \cdot \frac{\gcd(t,d)}{d} \right)^{2k}.
\end{align*}
Let us fix a prime $p \mid N$. Then for all $c,d \mid N$, we have
\begin{align}
T(c,d;N)=&\sum_{\substack{t \mid N/p}} (-1)^{\omega(t)+\omega(c)} \left( \frac{N\gcd(c,t)}{t} \cdot \frac{\gcd(t,d)}{d} \right)^{2k}\nonumber \\
& ~~+\sum_{\substack{t \mid N/p}} (-1)^{\omega(tp)+\omega(c)} \left( \frac{N\gcd(c,tp)}{tp} \cdot \frac{\gcd(tp,d)}{d} \right)^{2k}\nonumber \\
&=(p^{2k}-(\gcd(c,p)\gcd(d,p))^{2k})T(c,d;N/p), \label{eq:17}
\end{align}
where in the last step we use the equation (for $p \nmid t$)
\begin{align*}
 \frac{\gcd(c,tp)}{tp}\cdot \frac{\gcd(tp,d)}{d} = \frac{\gcd(c,p)\gcd(d,p)}{p} \cdot \frac{\gcd(c,t)}{t }\cdot \frac{\gcd(t,d)}{d}.
\end{align*}

If $c \neq d$, then there exists a prime $p \mid N$, such that  $\gcd(c,p)\gcd(d,p)=p$. Then by (\ref{eq:17}), we have $T(c,d;N)=0$.

For $c=d$, by (\ref{eq:17}), we have
\begin{align*}
T(c,c;N) &= T(c,c;1) \prod_{p \mid N} (p^{2k}-\gcd(c,p)^{4k})\\
&= (-1)^{\omega(c)} \left( \frac{1}{ c} \right)^{2k} \prod_{p \mid N} (p^{2k}-\gcd(c,p)^{4k})\\
&= \left( \frac{1}{ c} \right)^{2k} \prod_{p \mid c} (p^{4k}-p^{2k}) \prod_{p \mid N/c} (p^{2k}-1)\\
&=  \prod_{p \mid N} (p^{2k}-1).
\end{align*}

\end{proof} 

\section{Application: Convolution sums of the divisor function} \label{sec:4}

Let $N\in \mathbb{N}$ be square-free, $l,m >1$ be integers and $a,b \in \mathbb{N}$ be such that $\lcm(a,b) \mid N$. Then we have
\begin{align*}
E_{2l}(az)E_{2m}(bz) \in M_{2(l+m)}(\Gamma_0(N)).
\end{align*}
Thus, combining Theorem \ref{th:1} and (\ref{eq:19}), we obtain the following theorem.
\begin{theorem}\label{th:7}
Let $N\in \mathbb{N}$ be square-free. Let $l,m >1$ be integers and $a,b \in \mathbb{N}$ be such that $\lcm(a,b) \mid N$. Then there exists a cusp form $C_{N}(E_{2l}(az)E_{2m}(bz);z) \in S_{2k}(\Gamma_0(N))$ such that 
{\scriptsize \begin{align}
E_{2l}(az)E_{2m}(bz) = &\sum_{d \mid N} \left( \sum_{c \mid N}  \frac{(-1)^{\omega(d)+{\omega(c)}}}{\prod_{p \mid N} (p^{2k}-1)} \left( \frac{ N \gcd(c,d)}{c}  \right)^{2k} \left( \frac{\gcd(a,c)}{a} \right)^{2l} \left( \frac{\gcd(b,c)}{b} \right)^{2m}  \right) E_{2k}(dz) \nonumber \\
&+ C_{N}(E_{2l}(az)E_{2m}(bz);z), \label{eq:4}
\end{align} }%
where $k=l+m$.
Comparing coefficients of $q^n$ {\rm ($n>0$)} in {\rm (\ref{eq:4})} we obtain the formula for $W(a^{2l-1},b^{2m-1};n) $ given by {\rm (\ref{eq:31})}.
\end{theorem}
Below we give the description of cusp forms in $S_{2k}(\Gamma_0(6))$ for all $k>1$ in terms of eta quotients. Let $k,i \in \mathbb{N}$, and let us define the following eta quotient
\begin{align}
S(2k,6,i;z)=\left( \frac{\eta^6(z)\eta(6z)}{\eta^3(2z)\eta^2(3z)} \right)^{2k} \left( \frac{\eta(2z)\eta^5(6z)}{\eta^5(z)\eta(3z)} \right)^i \left( \frac{\eta^{13}(2z)\eta^{11}(3z)}{\eta^{17}(z)\eta^{7}(6z)} \right). \label{eq:28}
\end{align}
We note that the order of $S(2k,6,i;z)$ at $\infty$ is $i$, i.e. we have  $[n]S(2k,6,i;z)=0$ when $n<i$. This ensures linear independence and lower triangular shape of coefficients of these eta quotients. By \cite[Corollary 2.3, p. 37]{Kohler} and \cite[Theorem 1.64]{onoweb} we have $S(2k,6,i;z) \in S_{2k}(\Gamma_0(6))$ for all $k>1$ and $1\leq i \leq 2k-3$. Since the dimension of $S_{2k}(\Gamma_0(6))$ is $2k-3$ for all $k>1$, and $ S(2k,6,i;z)$  ($1\leq i \leq 2k-3$) are linearly independent, we obtain a basis for $S_{2k}(\Gamma_0(6))$. 
 
\begin{theorem} \label{th:6}
Let $k>1$ be an integer. Then the set of eta quotients 
\begin{align*}
\{ S(2k,6,i;z):  1\leq i \leq 2k-3   \}
\end{align*}
form a basis for $S_{2k}(\Gamma_0(6))$, i.e. letting $s(z) \in S_{2k}(\Gamma_0(6))$ we have
{ \begin{align*}
s(z)= \sum_{i=1}^{2k-3} b_i S(2k,6,i;z),
\end{align*}}%
where $b_i$ can be calculated recursively.
\end{theorem}
Theorem \ref{th:6} gives an example of a family of modular form spaces which is generated by eta quotients, a question raised by Ono in \cite{onoweb} and answered by Rouse and Webb in \cite{rouse}. Now we can express $C_{6}(E_{2l}(az)E_{2m}(bz);z)$ for $\lcm(a,b) \mid 6$ as linear combinations of $S(2k,6,i;z)$.


Below we give two beautiful examples. We first let $a,b=1$ in  (\ref{eq:31}). Then we have 
\begin{align}
 W(1^{2l-1},1^{2m-1};n)    =& \frac{-k B_{2l}B_{2m}}{4lm B_{2k}} \sigma_{2k-1}(n)
+ \frac{B_{2l}}{4l} \sigma_{2m-1}(n) + \frac{B_{2m}}{4m}  \sigma_{2l-1}(n) \nonumber \\
 &  + \frac{B_{2l}B_{2m}}{16lm} \sum_{i=1}^{2k-3} [n] b_i S(2k,6,i;z) \label{eq:21}
\end{align}
which, for $l,m>1$, is a similar expression to the one given by Ramanujan in \cite{19ramanujanocaf}. In (\ref{eq:21}) the cusp part vanishes when $l=2,m=2$; $l=2,m=3$; $l=2,m=5$; $l=3,m=4$, in agreement with Ramanujan's results in \cite{19ramanujanocaf}. 

Second, we let $a=2,b=3$ in (\ref{eq:31}), and $1<l=m$. Then we have
{  \begin{align}
W(2^{2l-1},3^{2l-1};n) =&  -\frac{B_{2l}^2}{B_{4l}}\cdot \frac{   \sigma_{4l-1}(n) + 2^{2l}\sigma_{4l-1}(n/2) +3^{2l} \sigma_{4l-1}(n/3) +  6^{2l}\sigma_{4l-1}(n/6) }{2l (3^{2l}+1)(2^{2l}+1)} \nonumber \\
&+\frac{B_{2l}}{4l} \sigma_{2l-1}(n/3)   + \frac{B_{2l}}{4l}  \sigma_{2l-1}(n/2) + \frac{B^2_{2l}}{16l^2} \sum_{i=1}^{4l-3} [n] b_i S(4l,6,i;z). \label{eq:26}
\end{align} }%
Replace $n$ by $1$ in (\ref{eq:26}) to see that the cusp part of (\ref{eq:26}) never vanishes. 

\section{Application: On eta quotients and representations by quadratic forms} \label{sec:7}
In this section we apply Theorem \ref{th:1} to eta quotients in $M_{2k}(\Gamma_0(N))$. We start with a special eta quotient which has applications to representations by quadratic forms. Then in the second part we give a general statement concerning eta quotients. We then discuss when these eta quotients has no cusp form component.

Let $N\in \mathbb{N}$ be an odd square-free number, $k>1$ and let $\displaystyle g(z)=\frac{\eta^2(z)}{\eta(2z)}$. Then for $r_\delta \in \mathbb{N}_0$ ($\delta \mid N$), not all zeros, we have 
\begin{align}
\prod_{ \delta \mid N} g^{8r_\delta}(\delta z) \in M_{2k}(\Gamma_0(2N)), \label{eq:35}
\end{align} 
where $k = 2 \sum_{ \delta \mid N} r_\delta > 0$. $2N$ is square-free, thus we can apply Theorem \ref{th:1} to (\ref{eq:35}).
\begin{theorem} \label{th:8}
Let $N \in \mathbb{N}$ be an odd square-free number. Let $r_\delta \in \mathbb{N}_0$ ($\delta \mid N$, not all zeros), and $k=2 \sum_{ \delta \mid N} r_\delta > 0$. Then there exists a cusp form $C_{2N}\left(\prod_{ \delta \mid N} g^{8r_\delta}(\delta z);z \right) \in S_{2k}(\Gamma_0(2N))$, such that
{ \begin{align*}
\prod_{ \delta \mid N} g^{8r_\delta}(\delta z) =& \sum_{d \mid 2 N} \left( \sum_{c \mid  N} \frac{(-1)^{\omega(d)+\omega(2c)} }{\prod_{p \mid 2N} (p^{2k}-1)} \left( \frac{  N\gcd(2c,d)}{c}  \right)^{2k} \prod_{ \delta \mid N} \left(\frac{\gcd(2c,\delta)}{\delta}\right)^{4 r_{\delta}}  \right) E_{2k}(dz) \nonumber \\
&+ C_{2N}\left(\prod_{ \delta \mid N} g^{8r_\delta}(\delta z);z \right).
\end{align*}}%
\end{theorem}
\begin{proof} Let $N \in \mathbb{N}$ be an odd square-free number, $r_\delta \in \mathbb{N}_0$ ($\delta \mid N$, not all zeros), and $k=2 \sum_{ \delta \mid N} r_\delta > 0$. We use \cite[Proposition 2.1]{Kohler} to compute
\begin{align*}
[0]_c g^{8}(\delta z)=
 \begin{cases} 
     \left(\frac{\gcd(c,\delta)}{\delta}\right)^4  & \mbox{ if $2 \mid c$} \\
     0  & \mbox{ if $2 \nmid c$}
   \end{cases}.
\end{align*}
Thus we have
\begin{align*}
[0]_c \prod_{ \delta \mid N} g^{8r_\delta}(\delta z) =\begin{cases} 
       \prod_{ \delta \mid N} \left(\frac{\gcd(c,\delta)}{\delta}\right)^{4 r_{\delta}} & \mbox{ if $2 \mid c$} \\
    0  & \mbox{ if $2 \nmid c$}
   \end{cases}.
\end{align*}
The result follows from Theorem \ref{th:1}, by putting the values of $\displaystyle [0]_c \prod_{ \delta \mid N} g^{8r_\delta}(\delta z)$ in (\ref{eq:2}).
\end{proof}
Below we apply Theorem \ref{th:8}, to give formulas concerning representations by quadratic forms. Following Ramanujan's notation, let us define
\begin{align*}
\varphi(z)=\sum_{n=-\infty}^{\infty} q^{n^2}.
\end{align*}
Then we have
\begin{align*}
\sum_{n=0}^{\infty} N(a_1^{r_1},a_2^{r_2},\ldots,a_{m}^{r_m}; n)q^n=\prod_{i=1}^m \varphi^{r_i}(a_i z).
\end{align*}
On the other hand, by \cite[(1.3.13)]{SprtRmnj},  we have $\varphi(z)=g(z+1/2)$,  (or equivalently $\varphi(q)=g(-q)$). Replacing $z$ by $z+1/2$ in Theorem \ref{th:8}, we have the following statement.

\begin{theorem} \label{th:9}
 Let $N \in \mathbb{N}$ be an odd square-free number, $r_\delta \in \mathbb{N}_0$, $\delta \mid N$, not all zeros, and $k=2 \sum_{ \delta \mid N} r_\delta > 0$. Then we have
{\small \begin{align}
\prod_{ \delta \mid N} \varphi^{8r_\delta}(\delta z) = &  \sum_{d \mid  N} \left( \sum_{c \mid  N} \frac{(-1)^{\omega(d)+\omega(c)} }{\prod_{p \mid 2N} (p^{2k}-1)} \left( \frac{ 2N\gcd(c,d)}{c}  \right)^{2k} \prod_{ \delta \mid N} \left(\frac{\gcd(c,\delta)}{\delta}\right)^{4 r_{\delta}}  \right) E_{2k}(2dz) \nonumber \\
&+\sum_{ {d \mid N}} \left( \sum_{c \mid  N} \frac{(-1)^{\omega(d)+\omega(2c)} }{\prod_{p \mid 2N} (p^{2k}-1)} \left( \frac{ N\gcd(c,d)}{c}  \right)^{2k} \prod_{ \delta \mid N} \left(\frac{\gcd(c,\delta)}{\delta}\right)^{4 r_{\delta}}  \right) E_{2k}(dz+1/2) \nonumber \\
&+ C_{2N}\left(\prod_{ \delta \mid N} g^{8r_\delta}(\delta z+1/2);z \right).  \label{eq:27}
\end{align} }%
We compare coefficients of $q^n$ on both sides of (\ref{eq:27}) to obtain the formula for $N(1^{8r_1},\ldots,N^{8r_N};n)$ given by {\rm (\ref{eq:40})}.
\end{theorem}
Now we illustrate Theorem \ref{th:9}, for the case $N=p$, $p$ a prime. Let $r_1,r_p \in \mathbb{N}_0$ and $k=2 (r_1+r_p) $. Then we have
\begin{align}
N(1^{8r_1},p^{8r_p};n) = & \frac{4k}{B_{2k}} \left( \frac{ (-1)^n(p^{4r_1} -1 )}{(2^{2k}-1)(p^{2k}-1)}   \sigma_{2k-1}(n) - \frac{2^{2k} (p^{4r_1}-1)}{(2^{2k}-1)(p^{2k}-1)} \sigma_{2k-1}(n/2) \right. \nonumber \\
&   \left. + \frac{ (-1)^n (p^{2k}- p^{4r_1} )}{(2^{2k}-1)(p^{2k}-1)} \sigma_{2k-1}(n/p) - \frac{ 2^{2k} (p^{2k} - p^{4r_1})}{(2^{2k}-1)(p^{2k}-1)}  \sigma_{2k-1}(n/{2p}) \right) \nonumber \\
& + [n] C_{2p}\left(\prod_{ \delta \mid p} g^{8r_\delta}(\delta z+1/2);z \right). \label{eq:39}
\end{align}
Letting $r_1=r_p=l$ in (\ref{eq:39}), we obtain
{\begin{align}
 N(1^{8l},p^{8l};n) =&  \frac{16l}{B_{8l}} \left(\frac{  { (-1)^n} \left(\sigma_{8l-1}(n)  + p^{4l} \sigma_{8l-1}(n/p) \right)- {2^{8l} } \left( \sigma_{8l-1}(n/2)  +p^{4l} \sigma_{8l-1}(n/{2p})\right) }{(2^{8l}-1)(p^{4l}+1)} \right) \nonumber \\
& + [n] C_{2p}\left(\prod_{ \delta \mid p} g^{8l}(\delta z+1/2);z \right). \label{eq:38}
\end{align}}%
Recently in \cite{cooperrmf}, Cooper, Kane and Ye obtained formulas for $N(1^{k},p^{k};n)$, valid for all $k \in \mathbb{N}$ and $p=3,7,11,23$. For $t \in \{1,p \}$, we have
{ \begin{align}
(-1)^{n} \sigma_{8l-1}(n/t) - 2^{8l} \sigma_{8l-1}(n/{2t}) = -\sigma_{8l-1}(n/t) + 2 \sigma_{8l-1}(n/{2t}) -2^{8l} \sigma_{8l-1}(n/{4t}). \label{eq:37}
\end{align}}%
We use (\ref{eq:37}) to show that the Eisenstein parts of the formula in \cite{cooperrmf} and of (\ref{eq:38}) agrees when $k=8l$. Our formula is valid for all primes, but fails when $ 8 \nmid k $.


We let $p=3$, $r_1=k$ and $r_3=0$ in (\ref{eq:39}), then by Theorem \ref{th:6} and (\ref{eq:39}) we have
{\footnotesize \begin{align*}
N(1^{8k},3^{0};n) =  \frac{-4k}{B_{2k}} \left( \frac{ (-1)^{n+1}}{2^{4k}-1}   \sigma_{4k-1}(n) + \frac{2^{4k}}{2^{4k}-1} \sigma_{4k-1}(n/2) \right)+ \sum_{i=1}^{4k-3} (-1)^n b_i [n]S(4k,6,i;z),
\end{align*}}
which, after an algebraic manipulation, with an equation similar to (\ref{eq:37}), agrees with the Ramanujan-Mordell formula, see  \cite{cooper,mordell,19ramanujanocaf}. 

In general Theorem \ref{th:1} can be applied to any eta quotient in $M_{2k}(\Gamma_0(N))$. The statement is as follows.
\begin{theorem} \label{th:10}
Let $N\in \mathbb{N}$ be square-free and $k>1$ be an integer. Let $r_\delta \in \mathbb{Z}$ $(\delta \mid N)$, not all zero. If $f(z) =\prod_{\delta \mid N} \eta^{r_\delta} (\delta z) \in M_{2k}(\Gamma_0(N))$ then for $n>0$, we have
{\footnotesize \begin{align*}
[n]f=&\frac{-4k}{B_{4k}}\sum_{d \mid N}  \left( \sum_{\substack{c \mid N, \\ v_c(f)=0}} \frac{(-1)^{\omega(d)+\omega(c)} }{\prod_{p \mid N} (p^{2k}-1)} \left( \frac{ N\gcd(c,d)}{c}  \right)^{2k}  \prod_{\delta \mid N} \frac{\nu(r_\delta,c) {\gcd(c,\delta)}^{r_\delta/2}}{\delta^{r_{\delta}/2}}  \right) \sigma_{2k-1}(n/d)\\
&+[n]C_N(f;z),
\end{align*}}%
where $\nu(r_\delta,c) \in \cc$ is a norm $1$ constant which can be determined by \cite[Theorem 1.7]{Kohler}.
\end{theorem}

Note that if $C_N(f;z)=0$ then $f(z)\in E_{2k}(\Gamma_0(N))$, thus Theorem \ref{th:10} determines Fourier coefficients of the eta quotient $f(z)$. One can use Theorem \ref{th:5}, (\ref{eq:43}) below and pigeonhole principle to show that there are only finitely many eta quotients in $ \bigcup_{k \in \mathbb{N}} E_{2k}(\Gamma_0(N))$ for each $N \in \mathbb{N}$ square-free. We further believe, disregarding the repetitions, the following well-known equations are the only such examples:
\begin{align}
& \displaystyle \frac{\eta^{16}(z)}{\eta^{8}(2z)}=\displaystyle 1+\sum_{n>0}(-16\sigma_3(n)+256\sigma_3(n/2))q^{n}, \label{eq:44} \\
& \displaystyle \frac{\eta^{16}(2z)}{\eta^{8}(z)}=\displaystyle \sum_{n>0}(\sigma_3(n)-\sigma_3(n/2))q^{n}. \label{eq:45}
\end{align}
Below we try to explain the reason of this. Before we start, note that there are two types of repetitions, first if $f(z) \in E_{2k}(\Gamma_0(N))$, then we also have $f(z) \in E_{2k}(\Gamma_0(N'))$ for any $N \mid N'$. Second, if we have $f(z) \in E_{2k}(\Gamma_0(N))$, then $f(dz) \in E_{2k}(\Gamma_0(dN))$. Note that it is slightly different from the concept of oldforms (which is defined for cusp forms). In the following arguments we assume $f(z)$ is not a repetition. Assuming $f(z) =\prod_{\delta \mid N} \eta^{r_\delta} (\delta z) \in E_{2k}(\Gamma_0(N))$, we have
\begin{align}
f(z)=\sum_{d \mid N } a_d E_{2k}(dz). \label{eq:41}
\end{align}
We compare both sides of (\ref{eq:41}) at different cusps using (\ref{eq:19}) and (\ref{eq:34}). The zeros of the eta quotient yield to the equations
\begin{align}
0=[n]_c\sum_{d \mid N } a_d E_{2k}(dz) \mbox{, for $n < v_{c}(f)$   } . \label{eq:46}
\end{align}
On the other hand, for sum of orders of zeros of $f(z)$, we have
{ \begin{align}
S(k,N)=\sum_{r \in R(N)} v_{1/r}(f(z)) & =\sum_{c \mid N} \frac{N}{24\gcd(c^2,N)} \sum_{\delta \mid N} \frac{\gcd(c,\delta)^2 r_{\delta}}{ \delta} = \frac{k}{6} \prod_{p \mid N} (p+1), \label{eq:43}
\end{align}}%
see \cite[(4.2.9)]{ayginthesis} for the details, we also have the number of inequivalent cusps of $\Gamma_0(N)$ is $\vert R(N) \vert = \sigma_0(N)$. It appears, if $S(k,N) > \vert R(N) \vert$, then the number of linearly independent equations coming from (\ref{eq:46}) which equal to zero are equal to the number of variables $a_d$. This forces $f(z)=0$. We have $S(k,N) \leq \vert R(N) \vert$ for the couples $(k,N)=$ $(2,1)$, $(2,2)$, $(2,3)$, $(2,5)$, $(2,6)$, $(3,1)$, $(3,2)$, $(3,3)$, $(4,1)$, $(4,2)$, $(5,1)$, $(5,2)$. Finally, we use similar arguments used in \cite{alacaaygin2015} to find out (\ref{eq:44}) and (\ref{eq:45}) are the only eta quotients in spaces corresponding to the couples above.


\section{On non-square-free level modular forms} \label{sec:8}

In this section we describe how to determine the Eisenstein terms of any modular form in integer weight spaces. Serre's result \cite{SerreD} (or see \cite[p. 300]{onodistr}) is a widely used tool to derive arithmetic properties of modular forms. However this requires the modular form to be a cusp form. Below we explain how to strip the modular forms from their Eisenstein part to determine their cusp part. Then using arithmetic properties of Eisenstein series and Serre's result one can derive arithmetic properties of modular forms (not necessarily a cusp form).

Let $M_{k}(\Gamma_0(N),\chi)$ be the space of integer weight modular forms for level $N$ (not necessarily square-free) with multiplier $\chi$. It is well-known that the basis of $E_{k}(\Gamma_0(N),\chi)$ can be given by the non-normalized Eisenstein series
\begin{align}
E_{k}(Mdz; \epsilon,\psi)=\sum_{\substack{m,n=-\infty,\\m,n\neq 0}}^{\infty} \frac{\epsilon(m)\psi(n)}{(Mdmz+n)^k}, \label{eq2:1}
\end{align}
where $d \mid \frac{N}{LM}$; and $\epsilon$ and $\psi$ are primitive Dirichlet characters with conductors $L$ and $M$, respectively, such that $LM \mid N$, $\psi \epsilon=\chi$ and $\psi \epsilon(-1)=(-1)^k$, see \cite[Ch. 7]{miyake} or \cite[Theorem 5.9]{stein} for details. To give the equivalent of Theorem \ref{th:1} for $f(z) \in M_{k}(\Gamma_0(N),\chi)$ one needs to compute the first terms of $E_{k}(Mdz; \epsilon,\psi)$ at the cusps of $\Gamma_0(N)$, and follow the arguments of the proof of the main theorem. Below we explain how to compute the first term of Fourier expansion of $E_{k}(Mdz; \epsilon,\psi)$ at cusp $r=\frac{\alpha}{\gamma}$ with $\gcd(\alpha,\gamma)=1$. Then there exist $ \beta,  \delta \in \mathbb{Z}$ such that $\alpha  \delta - \beta   \gamma =1$, i.e. we have  $R=\begin{bmatrix}    \alpha       & \beta  \\     \gamma     & \delta  \end{bmatrix} \in SL_2(\mathbb{Z})$. Hence the Fourier series expansion of $E_{k}(Mdz; \epsilon,\psi)$ at cusp $r$ is given by the Fourier series expansion of $E_{k}(Md R(z); \epsilon,\psi)$ at cusp $\infty$. When $k>2$ right hand side of (\ref{eq2:1}) is convergent, thus we have
\begin{align*}
E_{k}(MdR(z); \epsilon,\psi)&=\sum_{\substack{m,n=-\infty,\\ m,n\neq 0}}^{\infty} \frac{\epsilon(m)\psi(n)}{(MdmR(z)+n)^k}, \\
&=(\gamma z + \delta)^k \sum_{\substack{m,n=-\infty,\\m,n\neq 0}}^{\infty} \frac{\epsilon(m)\psi(n)}{((Mdm \alpha + \gamma n)z+(Mdm\beta +\delta n))^k}.
\end{align*}
Assuming the cusp width is $h$ and cusp parameter is $\kappa$, $E_{k}(MdR(z); \epsilon,\psi)$ has an expansion of the form $(\gamma z + \delta)^k \sum_{j \geq 0} a_j e^{2 \pi i (j+\kappa)z/h}$, see \cite[(1.18)]{Kohler}. Thus the contribution to the $a_0$ comes from the terms with $Mdm \alpha + \gamma n=0$, that is, for $k>2$, we have
\begin{align}
a_0&= \sum_{\substack{m,n=-\infty,\\m,n\neq 0, \\Mdm \alpha + \gamma n=0} }^{\infty} \frac{\epsilon(m)\psi(n)}{((Mdm \alpha + \gamma n)z+(Mdm\beta +\delta n))^k} \nonumber\\
&= \left( \frac{\gcd(Md,\gamma)}{Md}\right)^k \sum_{\substack{n=-\infty,\\ n\neq 0}}^{\infty} \frac{\epsilon \left(\frac{-\gamma }{\gcd(Md,\gamma)}n\right)\psi \left(\frac{Md \alpha}{\gcd(Md,\gamma)}n \right)}{ n^k}. \label{eq2:2}
\end{align}
This gives an equivalent of (\ref{eq:19}) for $E_{k}(Mdz; \epsilon,\psi)$. We then use these terms to form a set of linear equations as in (\ref{eq:11}) and solve it to obtain a statement equivalent to Theorem \ref{th:1} for $M_{k}(\Gamma_0(N),\chi)$. Noting that the expansion of $E_{k}(Mdz; \epsilon,\psi)$ at infinity is given by \cite[Theorem 7.1.3]{miyake}, the solution yields to a formula for the Eisenstein part of a modular form $f(z) \in M_{k}(\Gamma_0(N),\chi)$ with $2<k \in \mathbb{N}$ and $N \in \mathbb{N}$ not necessarily a square-free number.

Please note that the first terms of $E_{2k}(dz)$ at the cusps ${1}/{c}$ which can also be obtained by the above arguments agrees with the ones from Theorem \ref{th:3}. Theorem \ref{th:3} additionally gives all the Fourier coefficients of $E_{2k}(dz)$ at the cusps ${1}/{c}$, which is needed for some previous discussions. 

 Let $p>3$ be prime and $k=\frac{p-1}{2}$. Let  $\chi_p = \left( \frac{(-1)^kp}{*}\right)$ be the Kronecker symbol and  $\chi_1$ denote the primitive principal character. As an example we work on the eta quotient $\displaystyle \frac{\eta^p(pz)}{\eta(z)} \in M_k(\Gamma_0(p),\chi_p)$, a function which plays important role in the theory of partition functions. One can use the arguments above, \cite[Proposition 2.1]{Kohler}, and a well-known equation for the case $p=5$ to prove that there exists a cusp form $C_p(z) \in S_k(\Gamma_0(p),\chi_p)$ such that
\begin{align*}
\frac{\eta^p(pz)}{\eta(z)}=\frac{(-1)^k e^{2 \pi i (p-1)/8}}{2 L(k,\chi_p)} E_k(z;\chi_p,\chi_1) + C_p(z),
\end{align*}
where $L(k,\chi_p)$ is the Dirichlet L-function associated with $\chi_p$. Then we use the series expansion of $E_k(z;\chi_p,\chi_1)$ given by \cite[Theorem 7.1.3]{miyake} to obtain
\begin{align*}
\frac{\eta^p(pz)}{\eta(z)}= \frac{-(-1)^{\left \lfloor k/2 \right \rfloor }2k}{pB_{k,\chi_p}} \left( \sum_{n \geq 1} \sum_{d \mid n} \chi_p(n/d)) d^{k-1} q^n \right) + C_p(z),
\end{align*}
where $B_{k,\chi_p}$ is the generalized Bernoulli number, for which we have $\gcd(p,pB_{k,\chi_p})=1$, see \cite[Theorem 3]{carlitz}. It is easy to see $C_5(z)=0$ and $C_p \neq 0$ for $p>5$. For $p=5$ we have 
\begin{align*}
[n] \frac{\eta^5(5z)}{\eta(z)} \equiv 0 \pmod{5}, \mbox{ for $5 \mid n$}.
\end{align*}
For $p>5$, $p \mid n$ implies that $p \mid \sum_{d \mid n} \chi_p(n/d) d^{k-1}$. Hence we have
\begin{align}
[n] \frac{\eta^p(pz)}{\eta(z)} \equiv [n] C_p(z) \pmod{p}, \mbox{ for $p \mid n$}. \label{eq3:1}
\end{align}
This gives an alternative proof of $\displaystyle \sum_{n \geq 1}[np]\frac{\eta^p(pz)}{\eta(z)} q^{n} $ being equivalent to a cusp form with $v_{0} > p/24$ modulo $p$, one of the key arguments in Ono's celebrated paper `Distribution of the partition function modulo $m$', see \cite[Section 3]{onodistr}. 

\section*{Acknowledgments} 
I would like to thank Song Heng Chan for his helpful comments throughout the course of this research, and I am grateful to Heng Huat Chan, whose feedback helped to give a more elegant proof of Lemma \ref{lemma:1}. I am also indebted to Kenneth S. Williams for his encouraging remarks and corrections on an earlier version of this work. The author was supported by the Singapore Ministry of Education Academic Research Fund, Tier 2, project number MOE2014-T2-1-051, ARC40/14.

\bibliographystyle{amsplain}

\begin{thebibliography}{0}

\bibitem{alacaaygin2015}
{ A. Alaca, \c S. Alaca \and Z. S. Aygin,} 
\newblock {\it Fourier coefficients of a class of eta quotients of weight $2$,} 
\newblock Int. J. Number Theory {\bf 11} (2015), 2381--2392.

\bibitem{aygineisandconv}
{ Z. S. Aygin,} 
\newblock {\it Eisenstein series and convolution sums,} 
\newblock  Ramanujan J., accepted, doi:10.1007/s11139-018-0055-2.

\bibitem{ayginthesis}
{ Z. S. Aygin,}
\newblock {\it Eisenstein series, eta quotients and their applications in number theory},
\newblock  (Doctoral dissertation.) Carleton University, Ottawa, Canada. 2016.

\bibitem{SprtRmnj}
{ B. C. Berndt,} 
\newblock {\it Number Theory in the Spirit of Ramanujan,} 
\newblock Springer-Verlag, 1991.

\bibitem{carlitz}
{ L. Carlitz,} 
\newblock {\it Arithmetic properties of generalized Bernoulli numbers},  
\newblock Journal fur die reine und angewandte Mathematik (Crelles Journal), 1959.202 (2009): 174--182.

\bibitem{cooper}
{ H. H. Chan \and S. Cooper,} 
\newblock {\it Powers of theta functions},  
\newblock  Pacific J. Math. {\bf 235}(2008), 1--14.

\bibitem{cooperrmf}
{ S. Cooper, B. Kane \and D. Ye,}
\newblock {\it Analogues of the Ramanujan--Mordell theorem},
\newblock  J. Math. Anal. Appl. {\bf 446} (2017), 568--579.

\bibitem{iwaniec}
{ H. Iwaniec,}
\newblock {\it Topics in Classical Automorphic Forms,}
\newblock Graduate Studies in Mathematics, 17. American Mathematical Society, Providence, RI, 1997.

\bibitem{history2}
{ D. Kim, A. Kim \and A. Sankaranarayanan,}
\newblock {\it Convolution sums and their relations to Eisenstein series,}
\newblock Bull. Korean Math. Soc. {\bf 50}, (2013), No. 4, pp. 1389 -- 1413.

\bibitem{Kohler}
{ G. K\"{o}hler,}
\newblock {\it Eta Products and Theta Series Identities,}
\newblock  Springer Monographs in Mathematics, Springer, 2011.


\bibitem{miyake}
{T. Miyake,}
\newblock {\it Modular Forms,}
\newblock Springer-Verlag, Berlin (1989), translated from the Japanese by Yoshitaka Maeda.

\bibitem{mordell}
{ L. J. Mordell,} 
\newblock {\it On the representations of numbers as a sum of $2r$ squares},
\newblock  Quart. J. Pure and Appl. Math. {\bf 48} (1917), 93--104.

\bibitem{onodistr}
{ K. Ono,} 
\newblock {\it Distribution of the partition function modulo $m$,}
\newblock Annals of Mathematics, Second Series, Vol. 151, No. 1 (Jan., 2000), 293--307

\bibitem{onoweb}
{ K. Ono,} 
\newblock {\it The Web of Modularity: Arithmetic of the Coefficients of Modular Forms and q-Series,}
\newblock Am. Math. Soc., Providence, RI, 2004.

\bibitem{19ramanujanocaf}
{ S. Ramanujan,}
\newblock {\it On certain arithmetical functions,}
\newblock Trans. Cambridge Philos. Soc. 22 (1916) 159--184.

\bibitem{rouse}
{ J. Rouse, J. J. Webb,}
\newblock {\it On spaces of modular forms spanned by eta-quotients,}
\newblock Adv. Math. {\bf 272} (2015), 200--224.

\bibitem{SerreD}
{ J.-P. Serre,} 
\newblock {\it Divisibilit\'{e} de certaines fonctiones arithm\'{e}tiques,}
\newblock L'Ensein. Math. {\bf 22} (1976), 227--260.

\bibitem{Serre}
{J.-P. Serre,} 
\newblock {\it Modular forms of weight one and Galois representations,}
\newblock Algebraic number fields: L-functions and Galois properties (Proc. Sympos., Univ. Durham, Durham, 1975), Academic Press, London, 1977, 193--268.

\bibitem{stein}
{ W. A. Stein,}
\newblock {\it Modular Forms, A Computational Approach,}
\newblock Amer. Math. Soc., Graduate Studies in Mathematics {\bf 79} (2007).

\end{thebibliography}

\end{document}